\documentclass[10pt]{amsart}
\usepackage{amsmath,amssymb}

\usepackage{color,geometry,url}

\theoremstyle{plain}
\newtheorem{theorem}{Theorem}[section]

\newtheorem{lemma}[theorem]{Lemma}
\newtheorem{corollary}[theorem]{Corollary}
\theoremstyle{definition}
\newtheorem{definition}{Definition}[section]
\theoremstyle{remark}
\newtheorem{remark}{Remark}[section]

\numberwithin{equation}{section}

\newcommand{\AAi}{\mathcal A_1}
\newcommand{\AAm}{\mathcal A_m}
\newcommand{\UU}{\mathcal U}

\newcommand{\CC}{\mathbb C}
\newcommand{\RR}{\mathbb R}

\def\supp{\mathop{\operatorname{supp}}}

\makeatletter 
\newcommand{\LeftEqNo}{\let\veqno\@@leqno}

\DeclareFontFamily{U}{matha}{\hyphenchar\font45}
\DeclareFontShape{U}{matha}{m}{n}{
      <5> <6> <7> <8> <9> <10> gen * matha
      <10.95> matha10 <12> <14.4> <17.28> <20.74> <24.88> matha12
      }{}
\DeclareSymbolFont{matha}{U}{matha}{m}{n}
\DeclareFontSubstitution{U}{matha}{m}{n}

\DeclareFontFamily{U}{mathx}{\hyphenchar\font45}
\DeclareFontShape{U}{mathx}{m}{n}{
      <5> <6> <7> <8> <9> <10>
      <10.95> <12> <14.4> <17.28> <20.74> <24.88>
      mathx10
      }{}
\DeclareSymbolFont{mathx}{U}{mathx}{m}{n}
\DeclareFontSubstitution{U}{mathx}{m}{n}
\DeclareMathAccent{\widecheck}{0}{mathx}{"71}
\DeclareMathAccent{\wideparen}{0}{mathx}{"75}
\DeclareMathDelimiter{\vvvert}{0}{matha}{"7E}{mathx}{"17}

\newcommand{\oscU}{\text{osc}_{\UU}}
\newcommand{\oscUG}{\text{osc}_{\UU,\Gamma}}
\newcommand{\oscUGd}{\text{osc}_{\UU^\delta,\Gamma}}
\newcommand{\CoY}{\text{Co}Y}
\newcommand{\CoYt}{\widetilde{\text{Co}}Y}
\newcommand{\cb}{\color{blue}}
\newcommand{\Cmu}{C_{m,\UU}}
\newcommand{\Yf}{Y^\flat}
\newcommand{\Yn}{Y^\natural}
\newcommand{\dy}{d\mu(y)}
\newcommand{\dx}{d\mu(x)}
\newcommand{\UP}{U_\Phi}
\newcommand{\SP}{S_\Phi}

\begin{document}

\author{Nicki Holighaus$^\dag$}
\address{$^\dag$ Acoustics Research Institute Austrian Academy of Sciences, 
Wohllebengasse 12--14, A-1040 Vienna, Austria}
\email{nicki.holighaus@oeaw.ac.at, peter.balazs@oeaw.ac.at}
\author{Peter Balazs$^\dag$}

\title[Discretization in generalized coorbit spaces]{Discretization in generalized coorbit spaces: extensions, annotations and errata for ``Continuous Frames, 
Function Spaces and the Discretization Problem'' by M. Fornasier and H. Rauhut} 

\date{Last updated: \today}

 \begin{abstract}
   During the process of writing the manuscript \cite{bahowi15}, the work \cite{fora05}
   by Fornasier and Rauhut was one of the major foundations of our results and, naturally, 
   we found ourselves going back to reading that contribution once and again.
   
   In particular in Section 5, which is concerned with the discretization problem, 
   we have found some typographical errors, small inaccuracies and some parts that we just would 
   have wished to be slightly more accessible. Finally, for our own theory, a generalization of a central definition required us to verify that all the derivations in 
   \cite[Section 5]{fora05} still hold after the necessary modifications.

   Considering the importance of the results in \cite{fora05} for the community, this was reason enough 
   to start this side project of re-writing that least accessible portion of Fornasier and Rauhut's 
   manuscript, eliminating the errors we found, adding annotations where we consider them useful and
   modifying the results to take into account the generalization we require for our own work.
   
   What you see is the result of our endeavor, a one-to-one substitute for \cite[Section 5]{fora05}.
\end{abstract}
  
  \maketitle
  
  \section{Disclaimer}
    The manuscript at hand is no complete scientific paper, but supposed to 
replace Section 5 ``Discrete Frames'' in \cite{fora05}, an open access version 
of which is available at \url{http://arxiv.org/abs/math/0410571}. We therefore 
use the notation and results in Sections 1-4 therein without re-introducing 
them. Where applicable, the text in \cite{fora05} will be reproduced 1:1, 
extended where it was deemed useful and modified where necessary. Changes 
and additions are contained in {\cb 
blue}, except for minor efforts to make notation more consistent. Major 
changes are justified in annotations. For definitions, results, equations and so 
on, we use the enumeration scheme from the arXiv version of \cite{fora05}, more precisely
\url{http://arxiv.org/abs/math/0410571v1}.\\

\emph{If you find this manuscript useful for your own work, we kindly ask you to cite
our associated contribution~\cite{bahowi15}, to which this manuscript can be considered 
an online addendum. If you happen to find any oversights of our own or you have any 
suggestions for improvement, please contact us through the e-mail address provided on 
the last page of this manuscript.}
    
\setcounter{section}{4}    
  \section{Discrete Frames}\label{sec:5}
    In this section we investigate conditions under which one can extract a 
discrete frame from the continuous one. In particular, we will derive atomic 
decompositions and \emph{Banach frames} for the associated coorbit spaces. 
    
    The basic idea is to cover the index set $X$ by some suitable covering $\UU 
= (U_i)_{i\in I}$ with countable index set $I$ such that the kernel $R$ does 
not 
``vary too much'' on each set $U_i$. This variation is measured by an auxiliary 
kernel\footnote{Even for simple examples, e.g. the kernel associated to the 
short-time Fourier transform, cf. \cite{bahowi15}, there is strong evidence 
that 
the conditions derived in this chapter cannot be fulfilled by the kernel 
$\oscU$ 
proposed in \cite{fora05}.} {\cb $\oscUG(x,y)$} associated to $R$. Choosing points 
$x_i\in U_i$, $i\in I$, we obtain a sampling of the continuous frame 
$\{\psi_x\}_{x\in X}$. Under certain conditions on $\oscUG$ the sampled system 
$\{\psi_{x_i}\}_{i\in I}$ is indeed a frame for $\mathcal H$, {\cb respectively 
a Banach frame (atomic decomposition) for $\CoY$ ($\CoYt$).}
    
    We start with a definition.
    
    \begin{definition}\label{def:5.1}
      A family $\UU = (U_i)_{i\in I}$ of subsets of $X$ is called 
\emph{(discrete) admissible covering} of $X$ if the following conditions are 
satisfied.
      \begin{itemize}
      \item Each set $U_i$, $i\in I$ is relatively compact and has non-void 
interior.
      \item It holds $X = \bigcup_{i\in I} U_i$.
      \item There exists some constant $N>0$ such that 
      \begin{equation}\label{eq:5.1}
	\sup_{j\in I} \#\{ i\in I,~U_i\cap U_j \neq \emptyset\} \leq N < 
\infty.
      \end{equation}
      \end{itemize}
      Furthermore, we say that an admissible covering $\UU = (U_i)_{i\in I}$ 
is 
\emph{moderate} if it fulfills the following additional conditions.
      \begin{itemize}
	\item There exists some constant $D>0$ such that $\mu(U_i)\geq D$ for 
all $i\in I$.
	\item There exists a constant $\widetilde{C}$ such that 
	\begin{equation}\label{eq:5.2}
	  \mu(U_i)\leq \widetilde{C}\mu(U_j)\quad \text{for all $i,j$ with } 
U_i\cap U_j \neq \emptyset.
	\end{equation}
      \end{itemize}
    \end{definition}
    
    Note that the index set $I$ is countable because $X$ is $\sigma$-compact. 
    We remark further that we do not require the size of the sets $U_i$ 
(measured with $\mu$) 
    to be bounded from above. We only require a lower bound. Condition 
\eqref{eq:5.2} means that
    the sequence $(\mu(U_i))_{i\in I}$ is $\UU$-moderate in the sense of 
\cite[Definition 3.1]{fegr85}.
    If the sets $U_i$ do not overlap at all, i.e. they form a partition, then 
this condition in satisfied 
    trivially. A recipe for the construction of more general admissible 
coverings with property \eqref{eq:5.2}
    is discussed in \cite{fe87} together with some relevant examples.
    
    For the aim of discretization we have to restrict the class of admissible 
weight functions (resp. the class of 
    function spaces $Y$). From now on we require that there exists a moderate 
admissible covering $\UU = (U_i)_{i\in I}$
    of $X$ and a constant $C_{m,\UU}$ such that 
    \begin{equation}\label{eq:5.3}
      \sup_{x,y\in U_i} m(x,y) \leq \Cmu\quad \text{for all } i\in I.
    \end{equation}
    
    Of course, the trivial weight $1$ has this property (provided of course 
that moderate admissible coverings exists), 
    so that unweighted $L^p(X)$-spaces are admitted. Moreover, if $w$ is a 
continuous weight on $X$, then property
    \eqref{eq:5.3} of its associated weight on $X\times X$ defined by (3.6) 
means that $w$ is $\UU$-moderate in the 
    terminology introduced by Feichtinger and Gr\"obner in \cite[Definition 
3.1]{fegr85}. 
    
    The next definition will be essential for the discretization problem.
    
    \begin{definition}\label{def:5.2}
      A frame $\mathcal F$ is said to possess property $D[\delta,m]$ if there 
exists a moderate admissible covering 
      $\UU = \UU^\delta = (U_i)_{i\in I}$ of $X$ {\cb and a phase function 
$\Gamma: X\times X \rightarrow \CC$ with
      $|\Gamma| = 1$} such that \eqref{eq:5.3} holds and such that the kernel 
{\cb $\oscUG$ defined by\footnote{See annotation 1.}
      \begin{equation*}
	\oscUG(x,y) := \sup_{z\in Q_y} |\langle 
S^{-1}\psi_x,\psi_y-\Gamma(y,z)\psi_z\rangle| = \sup_{z\in Q_y} 
|R(x,y) - \Gamma(y,z)R(x,z)|,
      \end{equation*}}
      where $Q_y := \bigcup_{i,y\in U_i} U_i$, satisfies
      \begin{equation}\label{eq:5.4}
	  {\cb\|\oscUG |\AAm\| < \delta.}
      \end{equation}
    \end{definition}
    
    We assume from now on that the frame $\mathcal F$ possesses at least 
property $D[\delta,1]$ for some $\delta>0$. 
    Furthermore, we only admit weight functions $m$ (resp. spaces $Y$) for 
which the frame has the property 
    $D[\delta,m]$ for some $\delta>0$. 
    
    \subsection{Preparations}
    
    Associated to a function space $Y$ and to a moderate admissible covering 
$\UU = (U_i)_{i\in I}$ we will 
    define two sequence spaces. Before being able to state their definition we 
have to make sure that characteristic 
    functions of compact sets are contained in $Y$.
    
    \begin{lemma}\label{lem:5.1}
      If $Q$ is an arbitrary compact subset of $X$ then the characteristic 
function of $Q$ is contained in $Y$.
    \end{lemma}
    \begin{proof}
      Assume that $F$ is a non-zero function in $Y$. Then by solidity we may 
assume that $F$ is positive. 
      Clearly, there exists a non-zero continuous positive 
kernel
      $L\in \AAm$. The application of $L$ to $F$ yields a non-zero positive 
    continuous function in $Y$ (by the assumption on 
$\AAm$)
    . Hence, there 
    exists a compact set $U$ with non-void interior such that $L(F)(x)>0$ for 
all $x\in U$. By compactness of
    $U$ and continuity of $L(F)$ there exists hence a constant $C$ such that 
$\chi_U(x) \leq CL(F)(x)$ for all 
    $x\in X$. By solidity $\chi_U$ is contained in $Y$. Now, we set $K(x,y) = 
\mu(U)^{-1}\chi_Q(x)\chi_U(y)$, which 
    clearly is an element of $\AAm$ by compactness of $Q$ and $U$. It holds 
$\chi_Q = K(\chi_U)$ and hence $\chi_Q\in Y$.
    \end{proof}
    
    Now we may define the spaces\footnote{In \cite{fora05}, the spaces and 
natural norms do not fit together. This is corrected by exchanging $\lambda_i$ 
for $|\lambda_i|$ in the sums. 
    See also \cite{raul11} for the correct definition.}
    {\cb\begin{align*}
      Y^\flat & := \Yf(\UU) := \{ (\lambda_i)_{i\in I},\ \|\sum_{i\in I} 
|\lambda_i|\chi_{U_i} | Y \| < \infty\},\\
      \Yn & := \Yn(\UU) := \{ (\lambda_i)_{i\in I},\ \|\sum_{i\in I} 
|\lambda_i|\mu(U_i)^{-1}\chi_{U_i} | Y \| < \infty\}
    \end{align*}}
    with natural norms
    \begin{align*}
      \| (\lambda_i)_{i\in I} | \Yf\| & := \|\sum_{i\in I} 
|\lambda_i|\chi_{U_i} | Y \|,\\
      \| (\lambda_i)_{i\in I} | \Yn\| & := \|\sum_{i\in I} 
|\lambda_i|\mu(U_i)^{-1}\chi_{U_i} | Y \|.
    \end{align*}
    
    If the numbers $\mu(U_i)$ are bounded from above (by assumption they are 
bounded from below) then the two
    sequence spaces coincide. Lemma \ref{lem:5.1} implies that the finite 
sequences are contained in $\Yf$ and $\Yn$.
    If the space $(Y,\|\cdot|Y\|)$ is a solid Banach function space, then 
$(\Yf,\|\cdot|\Yf\|)$ and $(\Yn,\|\cdot|\Yn\|)$
    are solid BK-spaces, i.e. solid Banach spaces of sequences for which 
convergence implies componentwise convergence (this can be 
    seen, for example, as a consequence of Theorem \ref{thm:5.2}(d) and the 
fact that $\Yf\subset \Yn$). Let us state some further properties of these
    spaces.
    
    \begin{theorem}\label{thm:5.2}
      \begin{itemize}
	\item[(a)] The spaces $(\Yf,\|\cdot|\Yf\|)$ and $(\Yn,\|\cdot|\Yn\|)$ 
are Banach spaces.
	\item[(b)] If the bounded functions with compact support are dense in 
$Y$, then the finite sequences
	are dense in $\Yf$ and $\Yn$.
	\item[(c)] Assume that $w$ is a weight function on $X$ such that its 
associated weight $m(x,y) =$\linebreak $\max\{w(x)/w(y),w(y)/w(x)\}$
	satisfies \eqref{eq:5.3}. For $Y = L^p_w(X,\mu)$, $1\leq p\leq \infty$, 
it holds $\Yf = \ell^p_{b_p}(I)$ and $\Yn = \ell^p_{d_p}(I)$ with equivalent 
norms with 
	\begin{equation*}
	  b_p(i) := \mu(U_i)^{1/p}\tilde{w}(i),\quad d_p(i) := 
\mu(U_i)^{1/p-1}\tilde{w}(i)
	\end{equation*}
	where $\tilde{w}(i) = \sup_{x\in U_i} w(x)$.
	\item[(d)] Suppose that \eqref{eq:5.3} holds for the weight function $m$ 
associated to $Y$ and denote\footnote{Recall $v(x) = v_z(x) = m(x,z)$ for some 
fixed $z\in X$. Also recall that $v_z$ and $v_y$ are equivalent weights for all 
$z,y\in X$.} $\tilde v(i) = \sup_{x\in U_i} v(x)$ and $r(i) = 
\tilde{v}(i)\mu(U_i)$. Then $\Yn$ is continuously embedded into 
$\ell^\infty_{1/r}(I)$.
      \end{itemize}
    \end{theorem}
    \begin{proof}
      The statements (a) and (b) are straightforward to prove.
      
      {\cb 
	For (c) note that $\left(\sum_{j=1}^N |x_i|\right)^p/\left(\sum_{j=1}^N |x_i|^p\right)\leq N^p$ for arbitrary $x_i\in\CC$
	and use the finite overlap property of $\UU = (U_i)_{i\in I}$ to see that 
	\begin{align*}
	  \| \sum_{i\in I} |\lambda_i|\chi_{U_i} |L^p_w\|^p & = \int_X \Big|\sum_{i\in I}|\lambda_i|\chi_{U_i}(y)\Big|^p w^p(y)\dy \\
	  & \leq N^p\sum_{i\in I} |\lambda_i|^p \int_{U_i} w^p(y)\dy \leq N^p \|(\lambda_i)_{i\in I}|\ell^p_{b_p}\|^p.
	\end{align*}
        Alternatively, invoke Lemma \ref{lem:5.4} below to obtain an estimate with a possibly different constant\footnote{$\| \sum_{i\in I} |\lambda_i|\chi_{U_i} |L^p_w\|^p \leq \sum_{i\in I} \int_{U_i}
        (\lambda_i^+)^p w^p(y)\dy \leq \|(\mu(U_i)\lambda_i^+)_{i\in I}|\Yn\| \leq \tilde{C}\|(\tilde{\lambda}_i^+)_{i\in I}|\Yn\| \leq \tilde{C}C\|(\lambda_i)_{i\in I}|\Yf\|$, where $\tilde{C}$ is the moderateness constant from \eqref{eq:5.2} and $\tilde{\lambda}_i = \mu(U_i)\lambda_i$.}. For the other direction
        \begin{align*}
          \|(\lambda_i)_{i\in I}|\ell^p_{b_p}\|^p & = \sum_{i\in I} |\lambda_i|^p \mu(U_i) \tilde{w}^p(i) \leq \Cmu^p \sum_{i\in I} |\lambda_i|^p (\inf_{x\in U_i}w(x))^p \int_X \chi_{U_i}^p(y) \dy\\
          & \leq \Cmu^p \int_X \sum_{i\in I}  |\lambda_i|^p  \chi_{U_i}^p(y) w^p(y) \dy \\
          & \leq \Cmu^p \int_X \Big|\sum_{i\in I}  |\lambda_i| \chi_{U_i}(y)\Big|^p w^p(y) \dy = \Cmu^p \| \sum_{i\in I} |\lambda_i|\chi_{U_i} |L^p_w\|^p.
        \end{align*}
        Thus we have proven $(L^p_w)^\flat = \ell^p_{b_p}$. For $(L^p_w)^\natural = \ell^p_{d_p}$ just substitute $\mu(U_i)^{-1}\lambda_i$ for $\lambda_i$ in the derivations above.
      }
      
      For (d) we {\cb need to show that $|\lambda_i|\leq C r(i) 
\|(\lambda_j)_{j\in I}|\Yn\|$ for all $i\in I$.} 
      Fix some $k\in I$ and define the kernel 
      \begin{equation}\label{eq:5.5}
	K_i(x,y) = \chi_{U_k}(x)\chi_{U_i}(y),\ i\in I.
      \end{equation}
      For any $i\in I$ we obtain
      \begin{equation*}
	|\lambda_i|\chi_{U_k} = K_i(|\lambda_i|\mu(U_i)^{-1} \chi_{U_i}) \leq 
K_i(\sum_{j\in I}|\lambda_j|\mu(U_j)^{-1} \chi_{U_j}).
      \end{equation*}
      By solidity of $Y$ we get
      \begin{align*}
	|\lambda_i|\|\chi_{U_k}|Y\| & \leq \|K_i(\sum_{j\in 
I}|\lambda_j|\mu(U_j)^{-1} \chi_{U_j})|Y\| \leq \|K_i|\AAm\|\|\sum_{j\in 
I}|\lambda_j|\mu(U_j)^{-1} \chi_{U_j}|Y\|\\
	& = \|K_i|\AAm\|\|(\lambda_j)_{j\in I}|\Yn\|.
      \end{align*}
      {\cb Since $k$ is fixed, it remains to show $\|K_i|\AAm\|\leq 
C r(i)$ to complete the proof.} Let us estimate the $\AAm$-norm of $K_i$. With 
$y_0\in U_k$ we obtain
      \begin{align*}
	\int_X |K_i(x,y)|m(x,y)\dy & \leq \chi_{U_k}(x) \int_{U_i} m(x,y)\dy 
\leq \mu(U_i) \sup_{x\in U_k}\sup_{y\in U_i} m(x,y)\\
	& \leq \mu(U_i) \sup_{y\in U_i} m(y_0,y) \sup_{x\in U_k} m(x,y_0) \leq 
C\Cmu \mu(U_i)\tilde{v}(i) = {\cb C\Cmu r(i)},
      \end{align*}
      {\cb where we used that $m$ is admissible and} that different choices of 
$z$ in the definition (3.7) of $v$ yield equivalent weights. Furthermore a 
similar computation yields
      \begin{equation*}
        \begin{split}
	{\cb \int_X |K_i(x,y)|m(x,y)\dx~} & {\cb\leq \chi_{U_i}(y)\int_{U_k} m(x,y)\dx} \leq C \Cmu \mu(U_k)\tilde{v}(i)\\
& \leq C \Cmu D^{-1}\mu(U_k)\mu(U_i)\tilde{v}(i)
        \end{split}
      \end{equation*}
      where $D$ is the constant in Definition \ref{def:5.1} of a moderate 
admissible covering {\cb and has been added to treat the case $\mu(U_i) < 1$}. 
Hence, 
      $\|K_i|\AAm\| \leq C' r(i)$ for some suitable constant $C'$ (note that 
$k$ is fixed). This proves the claim.       
    \end{proof}
    
    Let us investigate the dependence of the spaces $\Yf$ and $\Yn$ on the 
particular covering chosen.

    \begin{definition}\label{def:5.3}
      Suppose $\UU = (U_i)_{i\in I}$ and $\mathcal V = (V_i)_{i\in I}$ are two 
moderate admissible coverings of $X$ over the same index set $I$. 
      Assume that $m$ is a weight function on $X\times X$. The coverings $\UU$ 
and $\mathcal V$ are called $m$-equivalent if the following conditions are 
satisfied.
      \begin{itemize}
	\item[(i)] There are constants $C_1,C_2 >0$ such that $C_1\mu(U_i) \leq 
\mu(V_i) \leq C_2\mu(U_i)$ for all $i\in I$.
	\item[(ii)] There exists a constant $C'$ such that $\sup_{x\in 
U_i}\sup_{y\in V_i} m(x,y)\leq C'$ for all $i\in I$.
      \end{itemize}
    \end{definition}
    
    \begin{lemma}\label{lem:5.3}
      Let $m$ be the weight function associated to $Y$ and suppose that $\UU = 
(U_i)_{i\in I}$ and $\mathcal V = (V_i)_{i\in I}$ are $m$-equivalent 
      moderate admissible coverings over the same index set $I$. Then it holds 
$\Yf(\UU) = \Yf(\mathcal V)$ and $\Yn(\UU) = \Yn(\mathcal V)$ with equivalence 
of norms.
    \end{lemma}
    \begin{proof}
      Assume that $(\lambda_i)_{i\in I}$ is contained in $\Yf(\mathcal V)$. 
Observe that the term 
      \begin{equation*}
	\int_X \chi_{V_i}(y)\chi_{V_j}(y)\dy \mu(V_j)^{-1}
      \end{equation*}
      equals $1$ for $i=j$ and for fixed $i$ it is non-zero for at most $N$ 
different indices $j$ by the finite overlap property \eqref{eq:5.1}. We obtain
      \begin{align*}
	\lefteqn{\sum_{i\in I} |\lambda_i|\chi_{U_i}(x) \leq \sum_{i\in I} 
|\lambda_i| \sum_{j\in I} \chi_{U_j}(x) \int_X \chi_{V_i}(y)\chi_{V_j}(y)\dy 
\mu(V_j)^{-1}}\\
	&= \int_X \sum_{i\in I} |\lambda_i| \chi_{V_i}(y) \sum_{j\in I} 
\chi_{U_j}(x)\chi_{V_j}(y)\mu(V_j)^{-1}\dy = L(\sum_{i\in I} |\lambda_i| 
\chi_{V_i})(x),
      \end{align*}
      where the kernel $L$ is defined by 
      \begin{equation}\label{eq:5.6}
	L(x,y):= \sum_{j\in I} \chi_{U_j}(x)\chi_{V_j}(y)\mu(V_j)^{-1}.
      \end{equation}
      The interchange of summation and integration is always allowed since by 
the finite overlap property the sum is always finite for fixed $x,y$. We claim 
that $L$ is contained in $\AAm$. Using property 
      {\cb (ii)} of $m$-equivalent coverings and once more the finite overlap 
property, we get
      \begin{align*}
	\int_X L(x,y)m(x,y)\dy & = \sum_{j\in I} \chi_{U_j}(x) \int_X 
\chi_{V_j}(y)\mu(V_j)^{-1}m(x,y)\dy \\
	&{\cb \overset{(\textrm{(ii)}}{\leq}} C'\sum_{j\in I} \chi_{U_j}(x)\leq C'N\quad 
\text{for all } x\in X.
      \end{align*}
      With property (i) and (ii) in Definition \ref{def:5.3} we get
      \begin{align*}
	\lefteqn{\int_X L(x,y)m(x,y)\dx = \sum_{j\in I} 
\chi_{V_j}(y)\mu(V_j)^{-1} \int_X \chi_{U_j}(x)m(x,y)\dx} \\
	&{\cb \overset{\textrm{(ii)}}{\leq}} C'\sum_{j\in I} 
\chi_{V_j}(y)\mu(V_j)^{-1}\mu(U_j) {\cb \overset{\textrm{(i)}}{\leq}} C'C_1N\quad 
\text{for all } y\in X.
      \end{align*}
      Thus, $L\in \AAm$ and by solidity of $Y$ we conclude that 
      \begin{equation*}
	\|(\lambda_i)_{i\in I}|\Yf(\UU)\| \leq \|L(\sum_{i\in I} 
|\lambda_i|\chi_{V_i}|Y\| \leq \|L|\AAm\|\|(\lambda_i)_{i\in I}|\Yf(\mathcal 
V)\|.
      \end{equation*}
      Exchanging the roles of $\UU$ and $\mathcal V$ gives a reversed 
inequality and thus $\Yf(\UU) = \Yf(\mathcal V)$. Moreover, replacing
      $(\lambda_i)_{i\in I}$ by $(\mu(U_i)^{-1}\lambda_i)_{i\in I}$ shows that 
$\Yn(\UU) = \Yn(\mathcal V)$.
    \end{proof}
    
    For some $i\in I$ we denote $i^\ast := \{j\in I, U_i\cap U_j\neq 
\emptyset\}$. Clearly, this is a finite set with at most $N$ elements. The next 
Lemma states that the 
    sequence spaces $\Yn$ are $\UU$-regular in the sense of \cite[Definition 
2.5]{fegr85}.
    
    \begin{lemma}\label{lem:5.4}
      For $(\lambda_i)_{i\in I} \in \Yn$ let $\lambda_i^+ := \sum_{j\in i^\ast} 
\lambda_j$. Then there exists some constant $C>0$ 
      such that $\|(\lambda_i^+)_{i\in I}|\Yn\|\leq C\|(\lambda_i)_{i\in 
I}|\Yn\|$.
    \end{lemma}
    \begin{proof}
      By Proposition 3.1 in \cite{fegr85} we have to prove that any permutation 
$\pi~:~I\rightarrow I$ satisfying {\cb $\pi(i)\in i^\ast$} for all $i\in I$ 
induces a 
      bounded operator on $\Yn$, i.e. $\|(\lambda_{\pi(i)})_{i\in I}|\Yn\|\leq 
C'\|(\lambda_{i})_{i\in I}|\Yn\|$, {\cb for some $C'$ independent of $\pi$}.
      We define the kernel 
      \begin{equation*}
	K_{\pi}(x,y):= \sum_{i\in I} 
\mu(U_{\pi^{-1}(i)})^{-1}\chi_{U_{\pi^{-1}(i)}}(x)\chi_{U_i}(y).
      \end{equation*}
      It is easy to see that
      \begin{align*}
	K_{\pi}(\mu(U_{j})^{-1}\chi_{U_{j}})(x) & {\cb = 
\mu(U_{\pi^{-1}(j)})^{-1}\chi_{U_{\pi^{-1}(j)}}
	+ \sum_{i\in I\setminus\{j\}} 
\mu(U_{\pi^{-1}(i)})^{-1}\chi_{U_{\pi^{-1}(i)}}(x)\int_X 
\chi_{U_i}(y)\chi_{U_j}(y)\dy}\\
	& \geq \mu(U_{\pi^{-1}(j)})^{-1}\chi_{U_{\pi^{-1}(j)}}.
      \end{align*}
      This gives 
      \begin{align*}
	\sum_{i\in I} |\lambda_{\pi(i)}|{\cb \mu(U_i)^{-1}}\chi_{U_i}(x) & = 
\sum_{i\in I} |\lambda_{i}|\mu(U_{\pi^{-1}(i)})^{-1}\chi_{U_{\pi^{-1}(i)}}(x) \leq 
\sum_{i\in I} |\lambda_{i}|K_{\pi}(\mu(U_{i})^{-1}\chi_{U_{i}})(x)\\
	& = K_{\pi}(\sum_{i\in I} |\lambda_{i}|\mu(U_{i})^{-1}\chi_{U_{i}})(x).
      \end{align*}
      Provided $K_\pi$ is contained in $\AAm$ this would give the result by 
solidity of $Y$. So let us estimate the $\AAm$-norm of $K_\pi$. We have
      \begin{align*}
	\int_X K_\pi(x,y)m(x,y)\dx & = \int_X \sum_{i\in I} 
\mu(U_{\pi^{-1}(i)})^{-1}\chi_{U_{\pi^{-1}(i)}}(x)\chi_{U_i}(y)m(x,y)\dx\\
	& \leq \left(\sum_{i\in I} \chi_{U_i}(y)\right) \sup_{i\in I}\sup_{y\in 
U_i}\sup_{x\in \cup_{j\in i^\ast} U_j} m(x,y)\leq \Cmu^2 N.
      \end{align*}
      Hereby, we used {\cb the finite overlap property and} that for $y\in 
U_i,x\in U_j$ with $U_i\cap U_j\neq \emptyset$ and $z\in U_i\cap U_j$ it holds 
$m(x,y)\leq m(x,z)m(z,y)\leq \Cmu^2$ 
      by property \eqref{eq:5.3}. Furthermore by property \eqref{eq:5.2}, we 
obtain 
      \begin{align*}
	\int_X K_\pi(x,y)m(x,y)\dy & {\cb =} \int_X \sum_{i\in I} 
\mu(U_{\pi^{-1}(i)})^{-1}\chi_{U_{\pi^{-1}(i)}}(x)\chi_{U_i}(y)m(x,y)\dy\\
	& \leq \Cmu^2 \sum_{i\in I} 
\mu(U_{\pi^{-1}(i)})^{-1}\mu(U_i)\chi_{U_{\pi^{-1}(i)}}({\cb x}) 
\overset{\eqref{eq:5.2}}{\leq} \Cmu^2\widetilde{C} N.
      \end{align*}
      This completes the proof.
    \end{proof}
    
    We will further need a partition of unity (PU) associated to a moderate 
admissible covering of $X$, i.e. a family $\Phi = (\phi_i)_{i\in I}$ of 
measurable functions that satisfies 
    $0\leq \phi_i(x)\leq 1$ for all $x\in X$, $\supp(\phi_i)\subset U_i$ and 
$\sum_{i\in I} \phi_i(x) = 1$ for all $x\in X$. The construction of such a 
family $\Phi$ subordinate to a locally finite 
    covering of some topological space is standard, see also 
\cite[pp. 127-{\cb 128, Proposition 4.41}]{fo84}.

    We may apply a kernel $K$ also to a measure $\nu$ on $X$ by means of 
    \begin{equation*}
      K(\nu)(x) = \int_X K(x,y) d\nu(y).
    \end{equation*}
    We define the following space of measures\footnote{Our definition is 
formally slightly different, but clearly equivalent to the one in \cite{fora05}.},
    \begin{equation*}
      D(\UU,M,\Yn) := \{\nu \in M_{loc}(X),~ \|(|\nu|(U_i))_{i\in I}|\Yn\| < \infty\}
    \end{equation*}
    with norm 
    \begin{equation*}
      \|\nu|D(\UU,M,\Yn)\| := \|(|\nu|(U_i))_{i\in I}|\Yn\|,
    \end{equation*}
    where $M_{loc}$ denotes the space of complex Radon measures. Spaces of this 
kind were introduced by Feichtinger and Gr\"obner in \cite{fegr85} who called 
them decomposition spaces.
    We identify a function with a measure in the usual way. Then 
    \begin{equation*}
      D(\UU,L^1,\Yn) := \{F \in L^1_{loc},~ \|(\int_{U_i}|F(x)|\dx)_{i\in 
I}|\Yn\| < \infty\}
    \end{equation*}
    with norm $\|F|D(\UU,L^1,\Yn)\| := \|(\|\chi_{U_i} F|L^1\|)_{i\in I}|\Yn\|$ 
can be considered as a closed subspace of $D(\UU,M,\Yn)$.
    
    We have the following auxiliary result.
    
    \begin{lemma}\label{lem:5.5}
      \begin{itemize}
	\item[(a)] It holds $Y\subset D(\UU,L^1,(L^\infty_{1/v})^\natural)$ with 
continuous embedding.
	\item[(b)] Assume that the frame $\mathcal F$ has property $D[\delta,m]$ 
for some $\delta>0$ {\cb and let $\UU^\delta = (U_i)_{i\in I}$ be a corresponding\footnote{i.e. 
there is a $\Gamma:X\times X\rightarrow \CC$ with $|\Gamma|=1$ such that $\|\oscUGd|\AAm\|<\delta$.} moderate admissible covering of $X$}. Then for $\nu\in D(\UU^\delta,M,\Yn)$ it holds 
	$R(\nu)\in Y$ and $\|R(\nu)|Y\|\leq C\|\nu|D(\UU^\delta,M,\Yn)\|$.        
      \end{itemize}
    \end{lemma}
    \begin{proof}
      (a) Assume $F\in Y$ and let 
      \begin{equation*}
	H(x):= \sum_{i\in I} \|\chi_{U_i}F|L^1\|\mu(U_i)^{-1}\chi_{U_i}(x).
      \end{equation*}
      We need to prove $H\in L^\infty_{1/v}$. Fix $k\in I$. Since $Y$ is 
continuously embedded into $L^1_{loc}$ {\cb by assumption}\footnote{See condition 
(Y1) in \cite[Section 3]{fora05}.} there 
      exists a constant $C$ such that $\|\chi_{U_i}F|L^1\|\leq C\|F|Y\|$ for 
all $F\in Y$. 
      With $K_i{\cb (x,y)=\chi_{U_k}(x)\chi_{U_i}(y)}$ as in \eqref{eq:5.5} 
(and fixed $k\in I$)
      it holds\footnote{$\mu(U_k)^{-1}K_i^\ast(\chi_{U_k})(x) = 
\mu(U_k)^{-1}\int_X \chi_{U_k}(y)\chi_{U_k}(y)\chi_{U_i}(x)\dy = \chi_{U_i}(x)$} 
$\chi_{U_i} = \mu(U_k)^{-1}K_i^\ast(\chi_{U_k})$.
      It is shown in the proof of {\cb Theorem \ref{thm:5.2}(d)} that 
$\|K_i|\AAm\| \leq C'\mu(U_i)v(x_i)$ for some constant $C'>0$ and {\cb some} 
$x_i\in U_i$. We obtain\footnote{$\|K_i^\ast(\chi_{U_k})|F||L^1\| = \langle K^\ast_i(\chi_{U_k}),|F|\rangle = \langle \chi_{U_k},K_i(|F|)\rangle
\|\chi_{U_k}K_i(|F|)|L^1\|$.}
      \begin{align*}
	\|\chi_{U_i}F|L^1\| & {\cb = \|\chi_{U_i}|F||L^1\| = 
\mu(U_k)^{-1}\|K_i^\ast(\chi_{U_k})|F||L^1\| = 
\mu(U_k)^{-1}\|\chi_{U_k}K_i(|F|)|L^1\|} \\
	& {\cb \leq C\mu(U_k)^{-1}\|K_i(|F|)|Y\| \leq 
C\mu(U_k)^{-1}\|K_i|\AAm\|\||F||Y\| = C''\mu(U_i)v(x_i)\||F||Y\|}\\
	& = C''\mu(U_i)v(x_i)\|F|Y\|
      \end{align*}
      where we used solidity of $L^1$ and $Y$. With this we obtain 
      \begin{equation*}
	H(x) \leq C''\|F|Y\|\sum_{i\in I}\chi_{U_i}(x)v(x_i).
      \end{equation*}
      For fixed $x$ this is a finite sum over the index set $I_x = \{i\in I,~ 
x\in U_i\}$. It holds 
      \begin{equation*}
	\sup_{i\in I_x} v(x_i) \leq \sup_{i\in I_x} m(x_i,x)m(x,z) \leq \Cmu 
m(x,z) = \Cmu v(x)
      \end{equation*}
      by \eqref{eq:5.3}. This proves $H\in L^\infty_{1/v}$ and the embedding is 
continuous.
      
      (b) Let $\Phi = (\phi_i)_{i\in I}$ be a PU associated to $\UU$. Further, 
we denote $R_i(x,y):= \phi_i(y)R(x,y)$. Clearly
      we have $R(x,y) = \sum_{i\in I} R_i(x,y)$. We obtain
      \begin{equation}\label{eq:A}\tag{A}
	|R_i(\nu)(x)| = |\int_X R_i(x,y) d\nu(y)|\leq \int_{U_i} |R_i(x,y)| 
d|\nu|(y) \leq |\nu|(U_i) \|R_i(x,\cdot)\|_\infty.
      \end{equation}
      {\cb On the other hand, }
      \begin{equation*}
	\mu(U_i)\|R_i(x,\cdot)\|_\infty \leq \int_X \chi_{U_i}(y)\sup_{z\in 
U_i} |R(x,z)|\dy.
      \end{equation*}
      Since the frame $\mathcal F$ is assumed to have property $D[\delta,m]$ we 
obtain by definition of $\oscUGd$ that 
      {\cb
      \begin{align*}
	|R(x,z)| & = |\Gamma(y,z)R(x,z)| = |R(x,y) + \Gamma(y,z)R(x,z) - 
R(x,y)| \\
	& \leq |R(x,y)| + |R(x,y)-\Gamma(y,z)R(x,z)| \leq \oscUGd(x,y)+|R(x,y)| 
\text{ for all } z,y\in U_i.
      \end{align*}}
      This gives
      \begin{equation}\label{eq:B}\tag{B}
	\mu(U_i)\|R_i(x,\cdot)\|_\infty \leq \int_X \chi_{U_i}(y)(\oscUGd(x,y) + 
|R(x,y)|)\dy = (\oscUGd + |R|)(\chi_{U_i})(x).
      \end{equation}
      {\cb Combine \eqref{eq:A} and \eqref{eq:B} to find}
      \begin{align}
	\|R(\nu)|Y\| & = \|\sum_{i\in I} R_i(\nu)|Y\| \leq \|\sum_{i\in I} 
|\nu|(U_i)\mu(U_i)^{-1}(\oscUGd + |R|)(\chi_{U_i})|Y\|\nonumber\\
	& = \left\| (\oscUGd + |R|)\left(\sum_{i\in I} 
|\nu|(U_i)\mu(U_i)^{-1}\chi_{U_i}\right)|Y \right\|\nonumber\\
	& \leq (\|\oscUGd|\AAm\|+\|R|\AAm\|) \|\sum_{i\in I} 
|\nu|(U_i)\mu(U_i)^{-1}\chi_{U_i}|Y \|\nonumber\\
	& = (\|\oscUGd|\AAm\|+\|R|\AAm\|) \|\nu|D(\UU,M,\Yn)\| \label{eq:5.7}
      \end{align}
      {\cb where we used the finite overlap property of the covering $\UU$.} 
This proves the claim.
    \end{proof}
    
    Using this Lemma we may prove that the assumption made in Proposition 3.7 
holds in case that the general assumptions of this section 
    are true\footnote{i.e. $\mathcal F$ possesses at least property 
$D[\delta,1]$ for some $\delta > 0$ and $Y$ is such that $\mathcal F$ has 
    property $D[\delta,m]$ for some $\delta > 0$, where $m$ is the weight 
function associated to $Y$.}.
    
    \begin{corollary}\label{cor:5.6}
      If the frame $\mathcal F$ has property $D[\delta,m]$ then $R(Y)\subset 
L^\infty_{1/v}$ with continuous embedding. In particular, Proposition 3.7 holds.
    \end{corollary}
    \begin{proof}
      Suppose $F\in Y$. By Lemma \ref{lem:5.5}(a) it holds $F\in 
D(\UU^\delta,L^1,(L^\infty_{1/v})^\natural)$, {\cb where $\UU^\delta$ is a covering of $X$ such 
that $\|\oscUGd|\AAm\|<\delta$ for a suitable phase function $\Gamma$. Such $\UU^\delta$ and $\Gamma$ exist, 
since $\mathcal F$ has property $D[\delta,m]$} and by Lemma \ref{lem:5.5}(b) we get $R(F)\in L^\infty_{1/v}$.
    \end{proof}
    
    \subsection{Atomic Decompositions and Banach Frames}
    
    Let us give the definition of an atomic decomposition and of a Banach 
frame. For a Banach space $B$ we denote its dual by $B^\ast$. 
    
    \begin{definition}\label{def:5.4}
      A family $(g_i)_{i\in I}$ in a Banach space  is called an atomic 
decomposition for $B$ if {\cb for some countable index set $I$} there exist a 
BK-space $(B^\natural(I),\| \cdot|B^\natural\|)$, $B^\natural = B^\natural(I)$, 
and linear bounded functionals $(\lambda_i)_{i\in I} \subset B^\ast$ (not 
necessarily unique) such that 
      \begin{itemize}
	\item $(\lambda_i(f))_{i\in I} \in B^\natural$ for all $f\in B$ and 
there exists a constant $0<C_1<\infty$ {\cb independent of $f$} such that 
	\begin{equation*}
	  \|(\lambda_i(f))_{i\in I}|B^\natural\| \leq C_1\|f|B\|,
	\end{equation*}
	\item if $(\lambda_i)_{i\in I} \in B^\natural$ then\footnote{Note that 
	here $(\lambda_i)_{i\in I}$ is used to denote a sequence in the sequence 
	space, not a sequence of functionals in $B^\ast$.}
	$f = \sum_{i\in I} \lambda_i g_i \in B$ (with unconditional convergence 
in some suitable topology) and there exists a constant $0<C_2<\infty$ {\cb 
independent of $(\lambda_i)_{i\in I}$} such that 
	\begin{equation*}
	  \|f|B\| \leq C_2\|(\lambda_i)_{i\in I}|B^\natural\|,
	\end{equation*}
	\item $f = \sum_{i\in I} \lambda_i(f)g_i$ for all $f\in B$.
      \end{itemize}
      \end{definition}

      We remark that this is not a standard definition (and probably such is 
not available). For instance, Triebel uses the same terminology with a slightly 
different meaning \cite[p.59 and p.160]{tr92}. The next 
definition is due to Gr\"ochenig \cite{gr91}.
      
      \begin{definition}\label{def:5.5}
	Suppose $(B,\|\cdot|B)$ is a Banach space. A family $(h_i)_{i\in I} 
\subset B^\ast$ is called a Banach frame for $B$ if there exists a BK-space 
$(B^\flat,\|\cdot|B^\flat)$, $B^\flat=B^\flat(I)$, and a linear bounded 
reconstruction operator $\Omega: B^\flat\rightarrow B$ such that 
	\begin{itemize}
	  \item if $f\in B$ then $((h_i(f))_{i\in I}\in B^\flat$, and there 
exist constants $0 <C_1\leq C_2 <\infty$ {\cb independent of $f$} such that 
	  \begin{equation*}
	    C_1\|f|B\| \leq \|(h_i(f))_{i\in I}|B^\flat\| \leq C_2\|f|B\|,
	  \end{equation*}
	  \item $\Omega(h_i(f))_{i\in I} = f$ for all $f\in B$.
	\end{itemize}
      \end{definition}
      
      Clearly, these definitions apply also with $B^\ast$ replaced by the 
anti-dual $B\urcorner$. Now we are prepared to state the main result of this 
article.
      
      \begin{theorem}\label{thm:5.7}
	Assume that $m$ is an admissible weight. Suppose the frame $\mathcal 
F=\{\psi_x\}_{x\in X}$ possesses property $D[\delta,m]$ for some $\delta>0$ 
and {\cb let $\UU^\delta$ denote 
a corresponding moderate admissible covering of $X$ such 
that 
	\begin{equation}\label{eq:5.8}
	  \delta (\|R|\AAm\| + \max\{C_{m,\UU^\delta}\|R|\AAm\|,\|R|\AAm\| + \delta\}) \leq 1
	\end{equation}
	where $C_{m,\UU^\delta}$ is the constant in \eqref{eq:5.3}}. Choose points $(x_i)_{i\in I} \subset X$ 
	such that $x_i\in U_i$. Moreover assume that $(Y,\|\cdot|Y\|)$ is a Banach space 
	fulfilling properties (Y1) and (Y2).
	
	Then $\mathcal F_d:=\{\psi_{x_i}\}_{i\in I} \subset \mathcal K^1_v$ is 
both an atomic decomposition of $\CoYt$ with corresponding sequence space $\Yn$ 
and a Banach frame for $\CoY$ with corresponding 
	sequence space $\Yf$. Moreover, there exists a 'dual frame' $\widehat{\mathcal F_d}:=\{e_{i}\}_{i\in I} \subset \mathcal H^1_v$ such that
	\begin{itemize}
	  \item[(a)] we have the norm equivalences 
	  \begin{equation*}
	    \|f|\CoY\| \cong \|(\langle f,\psi_{x_i}\rangle)_{i\in I}|\Yf\|\quad \text{and}\quad \|f|\CoYt\| \cong \|(\langle f,e_i\rangle)_{i\in I}|\Yn\|,
	  \end{equation*}
	  \item[(b)] if $f\in \CoYt$ then
	  \begin{equation*}
	    f = \sum_{i\in I} \langle f,e_i\rangle\psi_{x_i}
	  \end{equation*}
	  with unconditional norm convergence in $\CoYt$ if the finite sequences are dense in $\Yn$ and with unconditional 
	  convergence in the weak-$\ast$ topology induced from $(\mathcal H^1_v)\urcorner$ otherwise.
	  \item[(c)] if the finite sequences are dense in $\Yf$, then for all $f\in\CoY$ it holds 
	  \begin{equation*}
	    f = \sum_{i\in I} \langle f,\psi_{x_i}\rangle e_i
	  \end{equation*}
	  with unconditional convergence in the norm of $\CoY$.
	\end{itemize}
      \end{theorem}
      
      Also discretizations of the canonical dual frame lead to Banach frames and atomic decompositions.
      
      \begin{theorem}\label{thm:5.8}
	Under the same assumptions and with the same notation as in the previous theorem 
	$\tilde{\mathcal F}_d:=\{S^{-1}\psi_{x_i}\}_{i\in I} \subset \mathcal H^1_v$  is
	both an atomic decomposition of $\CoY$ (with corresponding sequence spaces $\Yn$) and
	a Banach frame for $\CoYt$ (with corresponding sequence space $\Yf$). Moreover, there
	exists a 'dual frame' $\widehat{\tilde{\mathcal{F}_d}}:=\{\tilde{e}_{i}\}_{i\in I} \subset \mathcal K^1_v$ 
	with the analogous properties as in the previous theorem.
      \end{theorem}
      
      Let us remark that the two previous theorems hold ``uniformly in Y''. Namely, if $m$ is fixed then the
      constant $\delta$ is the same for all function spaces $Y$ satisfying properties (Y1) and (Y2) with that
      specific $m$. In particular, the same covering $\UU^\delta = (U_i)_{i\in I}$ can be used for all those spaces $Y$ 
      and $(\psi_{x_i})_{i\in I}$, $x_i\in U_i$, is a Banach frame for all coorbit spaces $\CoY$ at the same time.
      
      The previous theorems imply an embedding result.
      
      \begin{corollary}\label{cor:5.9}
	We have the following continuous embeddings
	\begin{equation*}
	  \mathcal H^1_v \subset \CoY\subset (\mathcal K^1_v)\urcorner\quad \text{and}\quad \mathcal K^1_v \subset \CoYt\subset (\mathcal H^1_v)\urcorner
	\end{equation*}
      \end{corollary}
      \begin{proof}
	By {\cb definition\footnote{The reference to Proposition 3.7 and Corollary \ref{cor:5.6} in \cite{fora05} to show $Wf\in R(Y)\in L^\infty_{1/v}$ is superfluous.} $f\in\CoYt$ implies 
	$f\in (\mathcal H^1_v)\urcorner$ and the embedding $\CoYt\subset (\mathcal H^1_v)\urcorner$ is continuous 
	by Lemma 3.2 and Corollary \ref{cor:5.6}}. Lemma \ref{lem:5.1} shows that the Dirac element $\delta_i(j):= \delta_{i,j}$ is contained in $\Yn$ and 
	this in turn implies with Theorem \ref{thm:5.7}\footnote{By the second point of the definition of atomic decompositions.} that all $\psi_{x_i}, i\in I$, are contained in $\CoYt$ 
	with\footnote{The reasoning behind the $\ell^1_v$ estimate in \cite{fora05} is unclear. Note however that we obtain, with $K_i$ as in \eqref{eq:5.5}, 
	$\chi_{U_i} = \mu(U_k)^{-1}K_i^\ast(\chi_{U_k})$. By the estimates in the proof of Theorem \ref{thm:5.2}(d), we have $\|K_i|\AAm\|\leq \Cmu \max\{\mu(U_k),\mu(U_i)\}\sup_{x\in U_i}v(x)$.
	By the definition of $v$ and property \eqref{eq:5.3} of $m$, we obtain $\|\mu(U_i)^{-1}\chi_{U_i}|Y\| \leq \Cmu^2 D^{-1} \|\chi_{U_k}|Y\| v(x_i)$ for all $i\in I$ and some fixed $k\in I$.} 
	$\|\psi_{x_i}|\CoYt\|\leq {\cb C_2}\|\delta_i|\Yn\| {\cb = C_2\|\mu(U_i)^{-1}\chi_{U_i}|Y\| \leq  C C_2 v(x_i)}$. 
	Since any $x\in X$ may be chosen as one of the 
	$x_i$ it holds $\psi_x\in\CoYt$ for all $x\in X$ with $\|\psi_x|\CoYt\|\leq C'v(x)$. Corollary 3.4 
	hence implies that $\mathcal K^1_v$ is continuously embedded into $\CoYt$. The other 
	embeddings are shown analogously.
      \end{proof}
      
      We will split the proof of Theorems \ref{thm:5.7} and \ref{thm:5.8} into several lemmas. Let us just explain shortly the idea. Given a moderate admissible covering
      $\UU^\delta = (U_i)_{i\in I}$, a corresponding PU $(\phi_i)_{i\in I}$ and points $x_i\in U_i$, $i\in I$, we define the operator 
      \begin{equation}\label{eq:U}\tag{U}
	U_{\Phi}F(x) := \sum_{i\in I} c_i F(x_i)R(x,x_i)
      \end{equation}
      where $c_i = \int_X \phi_i(x)\dx$. Intuitively, $\UP$ is a discretization of the integral operator $R$.
      
      If $\UP$ is close enough to the operator $R$ on $R(Y)$ this implies that $\UP$ is invertible on $R(Y)$ since $R$ is the identity on $R(Y)$ by Proposition 3.7. Since $Wf\in R(Y)$ 
      whenever $f\in \CoYt$ and $R(x,x_i) = W(\psi_{x_i})(x)$ we conclude
      \begin{equation*}
	Wf = \UP\UP^{-1}Wf = \sum_{i\in I} c_i(\UP^{-1} Wf)(x_i)W(\psi_{x_i})
      \end{equation*}
      resulting in $f = \sum_{i\in I} c_i(\UP^{-1} Wf)(x_i)\psi_{x_i}$ by the correspondence principle stated in Proposition 3.7, {\cb once convergence is ensured}. 
      This is an expansion of an arbitrary $f\in \CoYt$ into the elements $\psi_{x_i}$, $i \in I$, and thus it gives a strong hint that we have in fact an atomic decomposition. 
      Reversing the order of $\UP$ and $\UP^{-1}$ and replacing $Wf$ by $Vf$
      {\cb 
      \begin{equation*}
	Vf = \UP^{-1}\UP Vf = \UP^{-1}\sum_{i\in I} c_i Vf(x_i)W(\psi_{x_i})
      \end{equation*}}
      leads to a recovery of an arbitrary $f\in \CoY$ from its coefficients $Vf(x_i) = \langle f,\psi_{x_i}\rangle$ and thus we may expect to have a Banach frame. 
      In the following we will make this rough idea precise. 
      In particular, we need to find conditions on $\delta$ that make sure that $\UP$ is close enough to the identity on $R(Y)$ (in fact this is ensured by \eqref{eq:5.8}). 
      Moreover, we will need some results that enable us to prove corresponding norm equivalences.
      
      Let us start with some technical lemmas.
      
      \begin{lemma}\label{lem:5.10}
	Suppose that the frame $\mathcal F$ has property $D[\delta,m]$ for some $\delta>0$ and that $\UU^\delta = (U_i)_{i\in I}$
	is a corresponding moderate admissible covering of $X$. Further, assume $(\lambda_i)_{i\in I}\in \Yn$ and $(x_i)_{i\in I}$ to be 
	points such that $x_i\in U_i$. Then $x\mapsto \sum_{i\in I} \lambda_i R(x,x_i)$ defines a function in $Y$ and 
	\begin{equation}\label{eq:5.9}
	  \|\sum_{i\in I} \lambda_i R(\cdot,x_i)|Y\| \leq C'\|(\lambda_i)_{i\in I}|\Yn\|.
	\end{equation}
	The convergence is pointwise, and if the finite sequences are dense in $\Yn$ it is also in the norm of $Y$. Furthermore, the series 
	$\sum_{i\in I} R(x,x_i)v(x_i)$ converges pointwise and absolutely to a function in $L^\infty_{1/v}$.
      \end{lemma}
      \begin{proof}
	Denote by $\epsilon_x$ the Dirac measure in $x$. then the application of $R$ to the measure $\nu:=\sum_{i\in I} \lambda_i\epsilon_{x_i}$
	results in the function $x\mapsto \sum_{i\in I} \lambda_i R(x,x_i)$. It follows from Lemma \ref{lem:5.4} that
	{\cb \begin{equation*}
	  \| \sum_{i\in I} \lambda_i\epsilon_{x_i}|D(\UU^\delta,M,\Yn)\| = \|\sum_{i\in I}|\nu|(U_i)|\Yn\| \leq \|(|\lambda_i|^+)_{i\in I} |\Yn\| \leq C\|(\lambda_i)_{i\in I} |\Yn\|
	\end{equation*}
	where $|\lambda_i|^+ = \sum_{j\in i^\ast} |\lambda_i|$ and $i^\ast = \{j\in I,~ U_i\cap U_j \neq \emptyset\}$.} Thus, Lemma \ref{lem:5.5}(b) yields \eqref{eq:5.9}. 
	If the finite sequences are dense in $\Yn$ then clearly the convergence is in the norm of $Y$.
	
	For the pointwise convergence {\cb in $L^\infty_{1/v}$} observe that the space $Y = L^\infty_{1/v}$ {\cb satisfies (Y2) with the associated weight function $m$}. 
	For this choice it holds
	$\Yn = \ell^\infty_{1/r}$ where $r(i) = v(x_i)\mu(U_i)$ (Theorem \ref{thm:5.2}(c))\footnote{Note that $d_p(i) = (\tilde{v}(i)\mu(U_i))^{-1}$, 
	where $\tilde{v}(i) = \inf_{x\in U_i} v(x)$, cp. Theorem \ref{thm:5.2}. Furthermore, $m$ is an admissible weight and hence 
	$v(x) = m(x,z) \leq m(x,x_i)m(x_i,z) \leq \Cmu v(x_i)$ and similarly $v(x_i) \leq \Cmu v(x)$ for all $x\in U_i$ by \eqref{eq:5.3}. 
	Therefore, $1/r(i)$ and $d_p(i)$ are equivalent weights.}.
	The application of $|R|$ to the measure $\nu = \sum_{i\in I} v(x_i)\mu(U_i)\epsilon_{x_i}$ yields $\sum_{i\in I} |R(\cdot,x_i)|v(x_i)\mu(U_i)$.
	The estimations in \eqref{eq:5.7} are also valid pointwise until the second line, yielding
	\begin{equation*}
	  R(\nu)(x)\leq (\oscUGd + |R|)(\sum_{i\in I} |\nu|(U_i)\mu(U_i)^{-1}\chi_{U_i})(x).
	\end{equation*}
	For our specific choice of $\nu$ we have
	\begin{equation*}
	  |\nu|(U_i) = \sum_{j, x_j\in U_i\cap U_j} |v(x_j)|\mu(U_i) \leq \sum_{j\in i^\ast} |v(x_j)|\mu(U_i) < \infty,
	\end{equation*}
	since this is a finite sum. Moreover, for fixed $x$ also 
	\begin{equation*}
	  H(x) = \sum_{i\in I}|\nu|(U_i)\mu(U_i)^{-1}\chi_{U_i}(x)
	\end{equation*}
	is a finite sum and hence converges pointwise. We already know that $H$ is contained in $L^\infty_{1/v}$, {\cb since obviously $(v(x_i)\mu(U_i))_{i\in I} = (r(i))_{i\in I} \in \ell^\infty_{1/r}$}. 
	We conclude that the partial sums of\linebreak $\sum_{i\in I} |R(x,x_i)|v(x_i)\mu(U_i)$ are dominated by 
	\begin{align}
	  \lefteqn{\int_X (\oscUGd + |R|)(x,y)H(y)\dy = \int_X (\oscUGd + |R|)(x,y)v(y)H(y)v^{-1}(y)\dy}\nonumber\\
	  & \leq \int_X (\oscUGd + |R|)(x,y)m(x,y)\dy m(x,z)\sup_{y\in X}\left(|H(y)|v^{-1}(y)\right) \label{eq:5.10}\\
	  & \leq m(x,z)\|\oscUGd + |R||\AAm\|\|H|L^\infty_{1/v}\|\nonumber
	\end{align}
	where we used the symmetry and property (3.2) of admissible weights. Hence, the sum\linebreak  $\sum_{i\in I} |R(\cdot,x_i)|v(x_i)\mu(U_i)$ converges pointwise.
	By Theorem \ref{thm:5.2}(d) we have $\Yn \subset \ell^\infty_{1/r}$ {\cb with continuous embedding} for general $Y$. 
	Together with the results just proven this yields that the convergence is also pointwise in general.
      \end{proof}

      \begin{lemma}\label{lem:5.11}
	Suppose that the frame $\mathcal F$ has property $D[\delta,m]$ for some $\delta >0$ and let
	$\UU^\delta = (U_i)_{i\in I}$ be an associated moderate admissible covering of $X$ with 
	corresponding PU $(\phi_i)_{i\in I}$. If $F\in R(Y)$ then for some constant $D>0$ it holds\footnote{We state the lemma in this modified, stronger form, since our form of \eqref{eq:5.11} is in fact invoked in both Corollary \ref{cor:5.12} and Theorem \ref{thm:5.13}, cp. \cite{fora05}. Note that the proof only requires trivial modification.}
	{\cb\begin{equation}\label{eq:5.11}
	   \|(F(x_i))_{i\in I}|\Yf\| \leq D\|F|Y\|\quad \text{and}\quad \|\sum_{i\in I}|F(x_i)|\phi_i|Y\| \leq \sigma\|F|Y\|
	\end{equation}}
        where $\sigma := \max\{C_{m,\UU^\delta}\|R|\AAm\|,\|R|\AAm\|+\delta\}$ with $C_{m,\UU^\delta}$ being the constant in \eqref{eq:5.3}. {\cb In particular, 
        \begin{equation*}
          \|\sum_{i\in I} F(x_i)\chi_{U_i}|Y\| \leq D\|F|Y\|\quad \text{and}\quad \|\sum_{i\in I}F(x_i)\phi_i|Y\| \leq \sigma\|F|Y\|,
        \end{equation*}
        by solidity.}
      \end{lemma}
      \begin{proof}
        {\cb We prove \eqref{eq:5.11}, the last part of the lemma then follows by solidity of $Y$ 
        or more specifically $\|F|Y\| = \||F||Y\|$.} Since $F\in R(Y)$ it holds $F=R(F)$ by
        Proposition 3.7 and Corollary \ref{cor:5.6}. This yields
        {\cb \begin{align*}
          H(x) & := \sum_{i\in I} |F(x_i)|\chi_{U_i}(x) = \sum_{i\in I} |R(F)(x_i)|\chi_{U_i}(x)
          = \sum_{i\in I_x}|\int_X R(x_i,y)F(y)\chi_{U_i}(x)\dy|\\
          & \leq \int_X \sum_{i\in I_x}|R(x_i,y)||F(y)|\chi_{U_i}(x)\dy.
        \end{align*}}
        Since the sum is finite over the index set $I_x = \{i\in I, x\in U_i\}$ the interchange of 
        summation and integration is justified. Define
         {\cb
	 \begin{equation}\label{eq:5.12}
	   K(x,y) := \sum_{i\in I} |R(x_i,y)|\chi_{U_i}(x)
         \end{equation}
         we obtain $H\leq K(|F|)$}. We claim that $K\in \AAm$. For the integral with respect to $y$ we 
        obtain
        \begin{equation*}
          \int_X |K(x,y)|m(x,y)\dy \leq \sum_{i\in I_x} \chi_{U_i}(x)m(x,x_i)\int_X |R(x_i,y)|m(x_i,y)\dy \leq NC_{m,\UU^\delta}\|R|\AAm\|
        \end{equation*}
        where {\cb we used property (3.2) of $m$ and} $N$ is the constant from \eqref{eq:5.1}. 
        For an estimation of the integral with respect to $x$ observe first that 
        {\cb
        \begin{align*}
          |R(x_i,y)| & \leq |\overline{\Gamma}(x,x_i)R(x_i,y) - R(x,y)| + |R(x,y)|\\
          & = |\Gamma(x,x_i)R(y,x_i) - R(y,x)| + |R(x,y)|\\
          & \leq \oscUGd(y,x) + |R(x,y)| = \oscUGd^\ast(x,y) + |R(x,y)|,
        \end{align*} }
        for all $x\in Q_{x_i} = \bigcup_{j, U_i\cap U_j\neq \emptyset} U_j$. By
        Fubini's theorem we obtain
        \begin{align*}
          \lefteqn{\int_X |K(x,y)|m(x,y)\dx = \int_X \sum_{i\in I} \chi_{U_i}(x)|R(x_i,y)|m(x,y)\dx}\\
          & \leq \sum_{i\in I} \int_{U_i} (\oscUGd^\ast(x,y) + |R(x,y)|)m(x,y)\dx \leq N\int_X (\oscUGd^\ast(x,y) + |R(x,y)|)m(x,y)\dx\\
          &\leq N(\|\oscUGd^\ast|\AAm\| + \|R|\AAm\|) < N(\|R|\AAm\|+\delta),
        \end{align*}
        since clearly $\|\oscUGd^\ast|\AAm\| = \|\oscUGd|\AAm\|$.
        This proves $K\in \AAm$ and we finally obtain 
        {\cb
        \begin{equation*}
          \|(F(x_i))_{i\in I}|\Yf\| = \|\sum_{i\in I} |F(x_i)|\chi_{U_i}|Y\| \leq \|K(|F|)|Y\| \leq \|K|\AAm\|\||F||Y\|=\|K|\AAm\|\|F|Y\|.
        \end{equation*}}
        A similar analysis shows also the second inequality in \eqref{eq:5.11}. The constant $N$ from \eqref{eq:5.1} does not enter the number $\sigma$ since we replace the characteristic
        functions by a partition of unity.
      \end{proof}
      
      \begin{corollary}\label{cor:5.12}
        Suppose the frame $\mathcal F$ possesses property $D[\delta,m]$ for some $\delta >0$. If $f\in\CoY$ then it holds $\|(Vf(x_i))_{i\in I}|\Yf\| \leq D\|f|\CoY\|$, {\cb where $D$ is the constant in \eqref{eq:5.11}}.
      \end{corollary}
      \begin{proof}
        By Proposition 3.7 it holds $Vf\in R(Y)$. {\cb By Lemma \ref{lem:5.11}} we conclude 
        $\|(Vf(x_i))_{i\in I}|\Yf\| \leq D\|Vf|Y\| = D\|f|\CoY\|$.
      \end{proof}
      
      As already announced we need to show that $\UP$ is invertible if $\delta$ is small enough.

      \begin{theorem}\label{thm:5.13}
        Suppose the frame $\mathcal F$ possesses property $D[\delta,m]$ for some $\delta >0$ {\cb 
        and let $\UU^\delta = (U_i)_{i\in I}$ be an associated moderate admissible covering of $X$
        with corresponding PU $(\phi_i)_{i\in I}$. Let $\UP$ be as in \eqref{eq:U}}. Then it holds
        \begin{equation}\label{eq:5.13}
          \|(\operatorname{Id} - \UP)|R(Y)\rightarrow R(Y)\| \leq \delta(\|R|\AAm\|+\sigma),
        \end{equation}
        where $\sigma$ is the constant from Lemma \ref{lem:5.11}. Consequently, $\UP$ is bounded and 
        if the right hand side of \eqref{eq:5.13} is less or equal to $1$ the $\UP$ is boundedly invertible on $R(Y)$.
      \end{theorem}
      \begin{proof}
        Let us first show the implicit assertion that $F\in R(Y)$ implies $\UP(F)\in R(Y)$. {\cb By Lemma \ref{lem:5.11} $(F(x_i))_{i\in I}\in \Yf$ which implies} $(c_iF(x_i))_{i\in I}\in \Yn$.
        It follows from Lemma \ref{lem:5.10} that $\sum_{i\in I} c_iF(x_i)R(\cdot,x_i)$ converges
        pointwise to a function $G=\UP(F)\in Y$. The pointwise convergence implies the 
        weak-$\ast$ convergence of $\sum_{i\in I} c_iF(x_i)\psi_{x_i}$ to an element $g$ of $(\mathcal H^1_v)\urcorner$ by Lemma 3.6(b) which is then automatically contained 
        in $\CoYt$ since $G\in Y$. From Lemma 3.6(c) follows that $G = Wg = R(Wg)$ and hence 
        $\UP(F)\in R(Y)$, {\cb implying $R(\UP(F))=\UP(F)$}.
        
      {\cb Let us now introduce the auxiliary operator\footnote{The rest of the proof has to be
      adjusted to take into account the modified definition of $\oscUG$.}
      \begin{equation*}
        \SP F(x) := R(\sum_{i\in I} \overline{\Gamma(\cdot,x_i)}F(x_i) \phi_i )(x),
      \end{equation*}
      where $\Gamma$ is the phase function from the definition of $\oscUG$.
      By the triangle inequality, 
      \begin{equation}\label{eq:T}\tag{T}
        \|F-\UP F|Y\|  \leq \|F-\SP F|Y\|  + \|\SP F-\UP F|Y\|.
      \end{equation}
      We now estimate both terms on the RHS separately. Assuming $F\in R(Y)$ implies $F = R(F)$ by Proposition 3.7 and Corollary \ref{cor:5.6}. This yields
      \begin{equation*}
      \begin{split}
        \|F-\SP  F|Y\| & = \|R(F - \sum_{i\in I} \overline{\Gamma(\cdot,x_i)}F(x_i)\phi_i)|Y\| \\ 
        & \leq \|R|\AAm\| \|\sum_{i\in I}\left(F-\overline{\Gamma(\cdot,x_i)}F(x_i)\right)\phi_i|Y\|.
      \end{split}
      \end{equation*}
      In order to estimate $\|\left(F-\overline{\Gamma(\cdot,x_i)}F(x_i)\right)\phi_i|Y\| $, examine
      \begin{equation*}
      \begin{split}
        \lefteqn{|\sum_{i\in I} \left(F(x)-\overline{\Gamma(x,x_i)}F(x_i)\right)\phi_i(x)|= |\sum_{i\in I} \left(R(F)(x)-\overline{\Gamma(x,x_i)}R(F)(x_i)\right)\phi_i(x)|}\\
        &= |\sum_{i\in I} \int_{X} F(y)\left(R(x,y)-\overline{\Gamma(x,x_i)}R(x_i,y)\right)~d\mu(y)\phi_i(x)|\\
        &\leq \sum_{i\in I} \int_{X} |F(y)||R(y,x)-\Gamma(x,x_i)R(y,x_i)|~d\mu(y)\phi_i(x)\\
        &\leq \sum_{i\in I} \int_{X} |F(y)| \oscUGd(y,x)~d\mu(y)\phi_i(x)\\   
        &= \sum_{i\in I} \oscUGd^\ast(|F|)(x)\phi_i(x) =  \oscUGd^\ast(|F|)(x).
      \end{split}
      \end{equation*}
      In the derivations above, we used $R(x,y) = \overline{R(y,x)}$ and the property $\supp(\phi_i)\subseteq U_i^\delta \in \mathcal{U}$ of the PU $\Phi=(\phi_i)_{i\in I}$. Furthermore, the interchange of summation and integration in the last line is allowed since by \eqref{eq:5.1} the sum is finite for any fixed $x\in X$.
      
      We obtain 
      \begin{equation}\label{eq:5.14}
        \|F-\SP  F|Y\|  \leq \|R|\AAm\| \|\oscUGd|\AAm\| \|F|Y\|,
      \end{equation}
       since $\|\oscUGd^\ast|\AAm\| = \|\oscUGd|\AAm\|$. 
      
      Now, we estimate $\|\SP F-\UP F|Y\| $. Note that
      \begin{equation*}
       \begin{split}
        \lefteqn{|\SP (F)(x) - \UP (F)(x)|}\\
        & = \left|\sum_{i\in I} \int_X F(x_i)\phi_i(y) \left(\overline{\Gamma(y,x_i)}R(x,y)-R(x,x_i)\right)~d\mu(y)\right|\\
        & \leq \sum_{i\in I} \int_X |F(x_i)|\phi_i(y) \left|R(x,y)-\Gamma(y,x_i)R(x,x_i)\right|~d\mu(y)\\
        & \leq \sum_{i\in I} \int_X |F(x_i)|\phi_i(y) \text{osc}_{\mathcal U^\delta,\Gamma}(x,y)~d\mu(y),
       \end{split}       
      \end{equation*}
      where we used $\supp(\phi_i)\subseteq U_i \in \mathcal{U}^\delta$ once more.
      
      Define $H(y):= \sum_{i\in I} |F(x_i)|\phi_i(y)$, then by Lemma \ref{lem:5.11} and solidity of $Y$:
      \begin{equation}\label{eq:5.15}
      \begin{split}
	\|\SP F-\UP F|Y\|  & \leq \|\text{osc}_{\mathcal U^\delta,\Gamma}|\AAm\| \|H|Y\| \\
	& \leq \sigma \|\text{osc}_{\mathcal U^\delta,\Gamma}|\AAm\| \|F|Y\|.
      \end{split}
      \end{equation}
      Insert \eqref{eq:5.14} and \eqref{eq:5.15} into \eqref{eq:T} and use $\|\text{osc}_{\mathcal U^\delta,\Gamma}|\AAm\| <\delta$ to complete the proof.}
      \end{proof}

      Now we have all the ingredients to prove Theorem \ref{thm:5.7}.
      
      \begin{proof}[Proof of Theorem \ref{thm:5.7}]
        The condition on $\delta$ implies by Theorem \ref{thm:5.13} that $\UP$ is invertible on $R(Y)$. Assuming $f\in \CoYt$ means $Wf\in R(Y)$ by
        {\cb Lemma 3.6(c)} and\footnote{The original proof in \cite{fora05} refers to Proposition 3.7(a) and Corollary \ref{cor:5.6}, which seems does not seem to be required here. However, Proposition 3.7(b) and Corollary \ref{cor:5.6} also imply $Wf\in R(Y)$ and will be used again in this context, see \eqref{eq:5.16}.} 
        {\cb the definition of $\CoYt$}. We conclude
        \begin{equation*}
          Wf(x) = \UP\UP^{-1}Wf(x) = \sum_{i\in I} c_i (\UP^{-1}Wf)(x_i) R(x,x_i) = \sum_{i\in I} c_i (\UP^{-1}Wf)(x_i) W\psi_{x_i}(x).
        \end{equation*}
        Setting $\lambda_i(f) := c_i (\UP^{-1}Wf)(x_i)$ we obtain with Proposition 3.7{\cb (b) and Corollary 5.6}
        \begin{equation}\label{eq:5.16}
          f = \sum_{i\in I} \lambda_i(f)\psi_{x_i}.
        \end{equation}
        Since $c_i \leq \mu(U_i)$ we obtain\footnote{Recall that $\vvvert \cdot|B\vvvert$ denotes the 
        operator norm in $\mathcal{B}(B)$.} with Lemma \ref{lem:5.11}
        \begin{align*}
          \|(\lambda_i(f))_{i\in I}|\Yn\| & \leq \|\left((\UP^{-1}Wf)(x_i)\right)_{i\in I}|\Yf\| \leq C\|\UP^{-1}Wf|Y\|\\
          & \leq C\vvvert \UP^{-1}|R(Y)\vvvert \|f|\CoYt\|,
          \end{align*}
        {\cb where we used $Wf\in R(Y)$ again.}
        Conversely, suppose that $(\lambda_i)_{i\in I}\in\Yn$ and form the function 
        \begin{equation*}
          H(x) := \sum_{i\in I} \lambda_i R(x,x_i) = \sum_{i\in I} \lambda_i W\psi_{x_i}(x).          
        \end{equation*}
        Since $\Yn {\cb \subset \ell^\infty_{1/r}}$ (Theorem \ref{thm:5.2}(d)) the sum converges pointwise to a function\footnote{Use Theorem \ref{thm:5.2}(c).} in $L^\infty_{1/v}$ by Lemma \ref{lem:5.10}. By Lemma 3.6(b) the pointwise convergence of the partial sums of $H$ implies the weak-$\ast$ convergence
        in $(\mathcal H^1_v)\urcorner$ of $f:=\sum_{i\in I} \lambda_i \psi_{x_i}$. Hence, $f$ is an element of $(\mathcal H^1_v)\urcorner$ and by Lemma \ref{lem:5.10} is therefore contained in 
        $\CoYt$. Also from Lemma \ref{eq:5.10} follows
        \begin{equation*}
          \|f|\CoYt\| = \|H|Y\| \leq C'\|(\lambda_i)_{i\in I}|\Yn\|
        \end{equation*}
        and the convergence of the sum representing $f$ is in the norm of $\CoYt$ if the finite sequences are dense in $\Yn$. This proves that $\mathcal F_d = (\psi_{x_i})_{i\in I}$ is an 
        atomic decomposition of $\CoYt$.
        
        Now suppose $f\in \CoY$ and let $F:=Vf\in R(Y)$. We obtain 
        \begin{equation}\label{eq:5.17}
          Vf = \UP^{-1}\UP Vf = \UP^{-1}\left(\sum_{i\in I} c_i Vf(x_i)W\psi_{x_i}\right).
        \end{equation}
        By the correspondence principle (Proposition 3.7) this implies
        \begin{equation*}
          f = {\cb V^{-1}}\UP^{-1}\left(\sum_{i\in I} c_i Vf(x_i)R(\cdot,x_i)\right).
        \end{equation*}
        This is a reconstruction of $f$ from the coefficients $Vf(x_i) = \langle f,\psi_{x_i}\rangle$, $i\in I$, and the reconstruction 
        operator $T~:~\Yf\rightarrow \CoY$, $T = V^{-1}\UP^{-1}J$ is bounded as the composition of bounded operators.
        Note that the operator\footnote{Note $(\lambda_i)_{i\in I}\in \Yf$ implies $(c_i\lambda_i)_{i\in I}\in\Yn$.} 
        $J((\lambda_i)_{i\in I})(x) := \sum_{i\in I}c_i\lambda_iR(x,x_i)$ is bounded by {\cb Lemma 5.10}. Setting $Y=L^{\infty}_{1/v}$ shows that 
        any element of $\text{Co}L^{\infty}_{1/v} = {\cb (\mathcal K^1_v)\urcorner}$ can be reconstructed in this way. 
        Now, if for $f\in (\mathcal K^1_v)\urcorner$ it holds $(\langle f,\psi_{x_i}\rangle)_{i\in I}\in \Yf$ then the series $\sum_{i\in I}\langle f,\psi_{x_i}\rangle\phi_i$ 
        converges to an element of $Y$ since $\phi_i\leq \chi_{U_i}$. By bounded invertibility of $\UP$ on $R(Y)$ the right hand side of \eqref{eq:5.17} defines an
        element in $Y$, hence $f\in\CoY$.
        
        Using \eqref{eq:5.17}, the norm equivalence follows from 
        \begin{align*}
          \|f|\CoY\| & = \|Vf|Y\| \leq \vvvert \UP^{-1}|R(Y)\vvvert \|\sum_{i\in I} c_i Vf(x_i)R(\cdot,x_i)|Y\|\\
          & \leq C\vvvert \UP^{-1}\vvvert \| (c_i Vf(x_i))_{i\in I}|\Yn\| \leq C\vvvert \UP^{-1}\vvvert \| (Vf(x_i))_{i\in I}|\Yf\| \leq C'\|f|\CoY\|.
        \end{align*}
        Hereby, we used Lemma \ref{lem:5.10}, $c_i\leq \mu(U_i)$ and Corollary \ref{cor:5.12}. Hence, we showed that $\mathcal F_d$ is a Banach frame for $\CoY$.
        
        In order to prove the existence of a dual frame let $E_i := c_i\UP^{-1}(W\psi_{x_i})\in R(L^1_v)$ and denote $e_i\in \mathcal H^1_v$ 
        the unique vector such that $E_i = V(e_i)$. If the finite sequences are dense in $\Yf$ then we may conclude from \eqref{eq:5.17} by a
        standard argument (see also\footnote{The original reference here is to Lemma 4, which does not exist in \cite{gr01}.} \cite[Lemma 5.3.2]{gr01})
        that $f=\sum_{i\in I} \langle f,\psi_{x_i}\rangle e_i$ with {\cb unconditional} norm convergence. This proves (c).
        
        We claim that 
        \begin{equation*}
          \lambda_i(f) = \langle f,e_i\rangle
        \end{equation*}
        yielding together with \eqref{eq:5.16} $f=\sum_{i\in I} \langle f,e_i\rangle\psi_{x_i}$ (with weak-$\ast$ convergence in general, and if the 
        finite sequences are dense in $\Yn$ with norm convergence). 
        
        If $F\in R(Y)$ then $F(x) = R(F)(x) = \langle F,W\psi_{x}\rangle$, {\cb by self-adjointness of $S^{-1}$}. 
        A simple calculation shows\footnote{We have $\langle R(\cdot,x_i),W\psi_{x}\rangle = \overline{\langle W\psi_{x},R(\cdot,x_i)\rangle} = \overline{W\psi_{x}(x_i)}$ and 
        $\overline{W\psi_{x}(x_i)}F(x_i) = \overline{W\psi_{x}(x_i)}R(F)(x_i) = \langle F, \UP W\psi_{x}\rangle$.}
        \begin{equation*}
          \langle \UP F,W\psi_{x}\rangle = \sum_{i\in I} c_i F(x_i) \langle R(\cdot,x_i),W\psi_{x}\rangle = \sum_{i\in I} c_i F(x_i) \overline{W\psi_{x}(x_i)} = \langle F,\UP W\psi_{x}\rangle.
        \end{equation*}
        Hence, the same relation applies to $\UP^{-1} = \sum_{i\in I} \sum_{n=0}^{\infty} (\operatorname{Id} - \UP)^n$ and we obtain
        \begin{align*}
          \lambda_i(f) & = c_i(\UP^{-1}Wf)(x_i) = c_i\langle \UP^{-1}Wf, W\psi_{x}\rangle = \langle Wf, c_i\UP^{-1}W\psi_{x}\rangle\\
          & = \langle Wf,Ve_i\rangle = \langle f,W^\ast V e_i \rangle = \langle f,e_i\rangle.
        \end{align*}
        {\cb Note that $W^\ast = V^{-1}$ on $\mathcal H^1_v \subseteq \mathcal H$.} By Lemma \ref{lem:5.10} and \ref{lem:5.11} we have the norm estimate
        \begin{align*}
          \|f|\CoY\| & = \|\sum_{i\in I} \langle f,e_i\rangle R(\cdot,x_i)|Y\| \leq C\|(\langle f,e_i\rangle)_{i\in I}|\Yn\| \leq C\|(\UP^{-1}Wf(x_i))_{i\in I}|\Yf\|\\
          &\leq C'\|\UP^{-1}Wf|Y\|\leq C' \vvvert\UP^{-1}\vvvert \|f|\CoY\|.
        \end{align*}
        This shows (a) and thus we complete the proof of Theorem \ref{thm:5.7}. Theorem \ref{thm:5.8} is proved in the same way by exchanging the roles of $V$ and $W$.
      \end{proof}
      
      \begin{remark}\label{rem:5.1}
        Using different approximation operators (compare{\cb \cite{gr91}}\footnote{The original reference here is to \cite{gr01} which does not seem to treat the discretization problem in detail.}) 
        one can prove that under some weaker condition on $\delta$ one may discretize the continuous frame in order to obtain only atomic decompositions or only Banach frames with no corresponding 
        results about (discrete) dual frames. In particular, if $\delta \leq 1$ then with the procedure of Theorem \ref{thm:5.7} one obtains atomic decompositions and in $\delta \leq \|R|\AAm\|^{-1}$
        one obtains Banach frames.
      \end{remark}
      
      Let us also add some comments about the Hilbert spaces situation which was the original question of Ali, Antoine and Gazeau. Here, we need to consider $Y=L^2$ since 
      $\text{Co}L^2 = \widetilde{\text{Co}}L^2 = \mathcal H$. By {\cb Theorem} \ref{thm:5.2}(c) the corresponding sequence space is 
      $\Yf = \ell^2_{\sqrt{a}}(I) = \ell^2(I,a)$ where $a_i =\mu(U_i)$. In order to be consistent with the usual notation of a (discrete) frame it seems suitable
      to renormalize the frame, i.e. under the conditions stated in Theorem \ref{thm:5.7} (according to Remark \ref{rem:5.1} it is only necessary to have $\delta\leq \|R|\AAm\|^{-1}$)
      it holds 
      \begin{equation*}
        C_1\|f|\mathcal H\|{\cb ^2} \leq \sum_{i\in I} |\langle f, \mu(U_i)^{1/2}\psi_{x_i}\rangle|^2 \leq  C_2\|f|\mathcal H\|{\cb ^2}.
      \end{equation*}
      This means that $(\mu(U_i)^{1/2}\psi_{x_i})_{i\in I}$ is a (Hilbert) frame in the usual sense. Of course, for the aim of Hilbert frames one may
      choose the trivial weight $m=1$ in Theorem \ref{thm:5.7}.
      
      One might ask whether the $L^1$-integrability condition $R\in \AAi$ is necessary in order to obtain a Hilbert frame by discretizing the continuous frame. 
      The crucial point in the proof of Theorem \ref{thm:5.7} is that the operator $\UP$ satisfies 
      \begin{equation}\label{eq:5.18}
        \|\UP-\operatorname{Id}|V(\mathcal H)\rightarrow V(\mathcal H)\| < 1.
      \end{equation}
      If one finds a method to prove this without using integrability assumptions on $R$ then the rest of the proof of Theorem \ref{thm:5.7} should still work. However, it is not clear
      to us how to do this in general.
      
      Concerning a complementary result F\"uhr gave the example of a continuous frame indexed by $\RR$ which does not admit a discretization by any regular grid of $\RR$ \cite[Example 1.6.9]{fu02-X}\footnote{The
      referenced result is not easily accessible using standard sources. However, the example referred to seems to be reproduced in \cite[Example 2.36]{fu05}}.
      
      {\cb This closes\footnote{We decided to omit Remark 5.2 from \cite{fora05} since sufficiency of the stated assumptions for a generalization of the discretization results to the setting of
      Remark 3.2 requires step-by-step confirmation. Such an endeavor is beyond the scope of this annotation.} Section \ref{sec:5}.}

\section*{Acknowledgment}
  This work was supported by the Austrian Science Fund (FWF)
  START-project FLAME (``Frames and Linear Operators for Acoustical
  Modeling and Parameter Estimation''; Y 551-N13) and the Vienna Science and 
Technology Fund (WWTF) 
  Young Investigators project CHARMED (``Computational harmonic analysis of 
high-dimensional biomedical data''; VRG12-009).
  

\bibliographystyle{abbrv}

\end{document}